\documentclass{article}

\usepackage{microtype}
\usepackage{faktor}
\usepackage{amsmath}
\usepackage{amssymb}
\usepackage{theorem}

\newtheorem{definition}{Definition}
\newtheorem{Theo}{Theorem}
\newtheorem{proposition}[Theo]{Proposition}
\newcommand{\cand}{\text{ and }}
\newenvironment{proof}[1][Proof]{\paragraph{{#1}}}%
                {{\hfill\(\Box\)\\}}
\newtheorem{corollary}[Theo]{Corollary}

\newcommand{\fall}[1]{{\forall\,{#1},\ }}
\newcommand{\fexist}[1]{{\exists\,{#1}\,{:}\ }}

\newcommand{\mb}[1]{{\bf #1}}

\newcommand{\mf}[1]{{\mathfrak #1}}

\newcommand{\tphi}{\mathop{\widetilde \varphi}}
\newcommand{\Ext}{\mathop{\mathbf{Ext}}\nolimits}

\DeclareMathOperator{\Fin}{\mathbf{Fin}}
\DeclareMathOperator{\Abs}{\mathbf{AC}}
\DeclareMathOperator{\Sem}{\mathbf{Co}}

\DeclareRobustCommand{\stirling}{\genfrac\{\}{0pt}{}}

\begin{document}


\title{A Purely Algebraic Summation Method}
\author{Olivier Brunet, Lyc\'ee Vaucanson, Grenoble\footnote{\texttt{olivier.brunet} \ at \ \texttt{normalesup.org}}}
\maketitle

It is mathematical folklore that
$$ 1 + 2 + 3 + 4 + \ \cdots \ = - \frac 1 {12} $$
This result is usually achieved using elaborate analytical methods, such as zeta function regularization or Ramanujan summation\cite{Hardy_Divergent}. However, in its notebooks, Ramanujan has also provided a very simple derivation which relied instead on algebraic manipulations. Recently, a video\footnote{\texttt{https://www.youtube.com/watch?v=w-I6XTVZXww}} from Numberphile has presented a similar derivation of the result (provoking lots of discussions and debates about the meaning of such an equality\footnote{See \texttt{https://www.nytimes.com/2014/02/04/science/in-the-end-it-all-adds-up-to.html?hpw\&rref=science} or \texttt{https://www.smithsonianmag.com/smart-news/great-debate-over-whether-1234-112-180949559/}}). It can be sketched as follows. Consider the infinite sums:
\begin{align*}
A & = 1 - 1 + 1 - 1 + 1 - 1 + \cdots \\
B & = 1 - 2 + 3 - 4 + 5 - 6 + \cdots \\
S & = 1 + 2 + 3 + 4 + 5 + 6 + \cdots
\end{align*}
We first have 
$$ A = 1 - (1 + 1 - 1 + 1 - 1 + 1 - 1 + \ \cdots) = 1 - A $$
so that $A = \frac 1 2$. Then,
$$ \begin{array}{r@{\ }c@{\ }r@{\ }c@{\ }r@{\ }c@{\ }r@{\ }c@{\ }r@{\ }c@{\ }r@{\ }c@{\ }r@{\ }c@{\ }r}
B - A & = & 1 & - & 2 & + & 3 & - & 4 & + & 5 & - & 6 & + & \cdots \\
& - & 1 & + & 1 & - & 1 & + & 1 & - & 1 & + & 1 & + & \cdots \\
& = & 0 & - & 1 & + & 2 & - & 3 & + & 4 & - & 5 & + & \cdots \\
& = & \rlap{\!\!\!$- B$}
\end{array} $$
so that $B = \frac 1 2 A = \frac 1 4$. Finally,
$$
\begin{array}{r@{\ }c@{\ }r@{\ }c@{\ }r@{\ }c@{\ }r@{\ }c@{\ }r@{\ }c@{\ }r@{\ }c@{\ }r@{\ }c@{\ }r}
S - 4 S & = & 1 & + & 2 & + & 3 & + & 4 & + & 5 & + & 6 & + & \cdots \\
& & & - & 4 & & & - & 8 & & & - & 12 & - & \cdots \\
& = & 1 & - & 2 & + & 3 & - & 4 & + & 5 & - & 6 & + & \cdots \\
& = & B
\end{array}
$$
which leads to the expected result:
$$ S = - \frac 1 3 B = - \frac 1 {12} $$
But this derivation, simple as it is, is usually considered as less rigorous than those using more elaborate analytical methods. One reason, in particular, is that in the derivation of the value of $S$, one needs to shift the terms of $-4 S$, an operation leading to potential difficulties.

However, this derivation is indeed perfectly rigourous, and in this article, we will define a general algebraic construction which we will use as a framework for expressing this derivation and, more generally, for providing a new summation method.

\section{An Algebraic Construction} \label{sec:construction}

In the following, $(\mf A, \otimes)$ will denot a unital commutative algebra over a field $\mb K$.

\begin{definition}
Let $M$ be a subalgebra of $\mf A$.
A vector subspace $F$ of $\mf A$ is \emph{$M$-stable}~if:
$$ M \subseteq F \qquad \hbox{and} \qquad \fall {m \in M, u \in F} m \otimes u \in F $$
Moreover, an $M$-form on $F$ is a linear form $\varphi: F \rightarrow \mb K$ such~that
$$ \varphi(1) = 1 \qquad \hbox{and} \qquad \fall {m \in M, u \in F} \varphi(m \otimes u) = \varphi(m) \varphi(u) $$
\end{definition}

Consider now an $M$-form $\varphi$ defined on an $M$-stable subspace $F$. Given $m \in M$ and $x \in \mf A$, if $\varphi(m) \neq 0$ and $m \otimes x \in F$, we define
$$ \widetilde \varphi_m(x) = \frac {\varphi(m \otimes x)}{\varphi(m)} $$
If $n \in M$ is also such that $\varphi(n) \neq 0$ and $n \otimes x \in F$, then
\begin{multline*}
\widetilde \varphi_n(x) = \frac {\varphi(n \otimes x)}{\varphi(n)} = \frac {\varphi(m) \varphi(n \otimes x)}{\varphi(m) \varphi(n)} = \frac {\varphi\bigl(m \otimes (n \otimes x)\bigr)}{\varphi(m \otimes n)} \\
= \frac {\varphi\bigl(n \otimes (m \otimes x)\bigr)}{\varphi(n \otimes m)} = \frac {\varphi(n) \varphi(m \otimes x)}{\varphi(n) \varphi(m)} = \frac {\varphi(m \otimes x)}{\varphi(m)} = \widetilde \varphi_m(x)
\end{multline*}
This observation suggests the following definition.
\begin{definition}
Given a subalgebra $M$ of $\mf A$, an $M$-stable subspace $F$ of $\mf A$ and an $M$-form $\varphi$ on~$F$, we define the $M$-extension of $F$ w.r.t.~$\varphi$~as
$$ \Ext_M(F, \varphi) = (\widetilde F, \widetilde \varphi) $$
where
$$ \widetilde F = \bigl\{ x \in \mf A \bigm| \fexist {m \in M} \varphi(m) \neq 0 \cand m \otimes x \in F \bigr\} 
$$
and, for all $x \in \widetilde F$, $\widetilde \varphi(x)$ is the common value of all the $\widetilde \varphi_m(x)$ for $m \in M$ such that $\varphi(m) \neq 0$ and $m \otimes x \in F$.
\end{definition}

The next result justifies the term ``extension'':

\begin{proposition} \label{prop:stability}
If $F$ is $M$-stable, then for $(\widetilde F, \widetilde \varphi) = \Ext_M(F, \varphi)$, we have $F \subseteq \widetilde F$ and $\widetilde \varphi|_F = \varphi$.
\end{proposition}
\begin{proof}
For all $x \in F$, $1 \otimes x \in F$ so that, as $\varphi(1) = 1 \neq 0$, $x \in \widetilde F$,~and
$$ \widetilde \varphi(x) = \widetilde \varphi_1(x) = \frac {\varphi(1 \otimes x)}{\varphi(1)} = \varphi(x) $$
\end{proof}

Moreover, clearly, if $F_1 \subseteq F_2$ are two $M$-stables subspaces, and if $\varphi$ is an $M$-form on $F_2$, then $\varphi|_{F_1}$ is an $M$-form on $F_1$ and for $(\widetilde F_1, \widetilde \varphi_1) = \Ext_M(F_1, \varphi|_{F_1})$ and $(\widetilde F_2, \widetilde \varphi_2) = \Ext_M(F_2, \varphi)$, we have $\widetilde F_1 \subseteq \widetilde F_2$ and $\widetilde \varphi_1 = \widetilde \varphi_2 |_{\widetilde F_1}$

\begin{proposition} \label{prop:linearity}
With the previous notations, $\widetilde F$ is a vector subspace of $\mf A$ and $\widetilde \varphi$ is linear.
\end{proposition}
\begin{proof}
Let $u, v \in \widetilde F$, and let $m, n \in M$ be such that $\varphi(m) \neq 0$, $\varphi(n) \neq 0$, $m \otimes u \in F$ and $n \otimes v \in F$. Let moreover $\lambda \in \mb K$.
One has
\begin{multline*}
(m \otimes n) \otimes (\lambda u + v) = (m \otimes n) \otimes \lambda u + (m \otimes n) \otimes v \\ = n \otimes (m \otimes \lambda u) + m \otimes (n \otimes v) = \lambda \bigl(n \otimes (m \otimes u)\bigr) + m \otimes (n \otimes v)
\end{multline*}
so that $\lambda u + v \in \widetilde F$. Moreover,
\begin{multline*}
\widetilde \varphi(\lambda u + v) = \frac {\varphi\bigl((m \otimes n) \otimes(\lambda u + v)\bigr)}{\varphi(m \otimes n)} = \frac{\varphi\bigl(n \otimes (m \otimes \lambda u) + m \otimes (n \otimes v)\bigr)}{\varphi(m \otimes n)} \\
= \lambda \frac {\varphi(n) \varphi(m \otimes u)}{\varphi(m) \varphi(n)} + \frac {\varphi(m) \varphi(n \otimes v)}{\varphi(m) \varphi(n)} = \lambda \, \widetilde \varphi(u) + \widetilde \varphi(v)
\end{multline*}
\end{proof}
As $M$ is stable by product, it is $M$-stable, so that one can define $\widetilde M = \Ext_M(M, \varphi|_M)$ (we drop the reference to the extension of $\varphi|_M$ as it is the restriction of $\tphi$ to $\widetilde M$).

\begin{proposition}
$\widetilde F$ is $\widetilde M$-stable and $\widetilde \varphi$ is an $\widetilde M$-form on $\widetilde F$.
\end{proposition}
\begin{proof}
Given $u \in \widetilde F$, $m \in \widetilde M$, we want to show that $m \otimes u \in \widetilde F$, i.e.\ that there exists an $n \in M$ such that $n \otimes m \otimes u \in F$. Let $a, b \in M$ be such that $\varphi(a) \neq 0$, $\varphi(b) \neq 0$, $a \otimes u \in F$ and $b \otimes m \in M$ and let $n = a \otimes b$. We have
$$ n \otimes (m \otimes u) = \underbrace{(b \otimes m)}_{\in M} \otimes \underbrace{(a \otimes u)}_{\in F} \in F $$
so that $m \otimes u \in \widetilde F$, as $\varphi(a \otimes b) \neq 0$. Now,
\begin{multline*}
\tphi(m \otimes u) = \frac {\varphi\bigl((a \otimes b) \otimes (m \otimes u)\bigr)}{\varphi(a \otimes b)} = \frac {\varphi\bigl((b \otimes m) \otimes (a \otimes u)\bigr)}{\varphi(a \otimes b)} \\ = \frac {\varphi(b \otimes m) \varphi(a \otimes u)}{\varphi(b)  \varphi(a)} = \tphi(m) \tphi(u)
\end{multline*}
\end{proof}

\begin{corollary}
$\widetilde M$ is a unital subalgebra of $\mf A$ and $\widetilde \varphi$ is an algebra homomorphism from $\widetilde M$ to $\mb K$.
\end{corollary}
\begin{proof}
It is $\widetilde M$-stable, so that it is stable by product in addition to being a vector subspace of $\mf A$. Similarly, $\widetilde \varphi$ is linear and preserves products.
\end{proof}

As $\widetilde F$ is $\widetilde M$-stable and $\widetilde \varphi$ is an $\widetilde M$-form on $\widetilde F$, one might want to consider the $\widetilde M$-extension of $\widetilde F$ w.r.t.\ $\tphi$. The next result shows that this is useless.

\begin{proposition}
If $({\widehat F}, {\widehat \varphi}) = \Ext_{\widetilde M}(\widetilde F, \widetilde \varphi)$, then $({\widehat F}, {\widehat \varphi}) = (\widetilde F, \widetilde \varphi)$.
\end{proposition}
\begin{proof}
It is sufficient to prove that $\widehat F \subseteq \widetilde F$. For any $x \in \widehat F$, there exists $m \in \widetilde M$ such that $m \otimes x \in \widetilde F$. But then, there exists $n \in M$ such that $n \otimes m \otimes x \in F$. Finally, as $m \in \widetilde M$, there exists $p \in M$ such that $p \otimes m \in M$. As a consequence, $p \otimes n \otimes m \otimes x \in F$. Now, $p \otimes m \in M$ so that $p \otimes n \otimes m \in M$ and hence $x \in \widetilde F$.

\end{proof}

\begin{proposition}[Cancellation Property]
If $m \in \widetilde M$ is such that $\widetilde \varphi(m) \neq 0$, then
$$ \fall {x \in \mf A} \quad x \in \widetilde F \iff m \otimes x \in \widetilde F $$
\end{proposition}
\begin{proof}
Obviously, as $\widetilde F$ is $\widetilde M$-stable, we have $x \in \widetilde F \implies m \otimes x \in \widetilde F$. Conversely, if $m \otimes x \in \widetilde F$, as $\widetilde \varphi(m) \neq 0$, we deduce that $x \in \widehat F$, i.e.\ $x \in \widetilde F$.
\end{proof}

Finally, we provide a simple criteria for proving that an element of $\mf A$ is not in $\widetilde M$.

\begin{proposition} \label{prop:not_in}
For all $x \in \mf A$, if there exists $m \in M$ such that $\varphi(m) = 0$, $m \otimes x \in F$ and $\varphi(m \otimes x) \neq 0$, then $x \not \in \widetilde F$.
\end{proposition}
\begin{proof}
Suppose otherwise, and let $n \in M$ such that $\varphi(n) \neq 0$ and $n \otimes x \in F$. One then has
$$ 0 = \varphi(m) \, \varphi(n \otimes x) = \varphi(m \otimes n \otimes x) = \varphi(n) \, \varphi(m \otimes x) \neq 0 $$
which is clearly absurd.
\end{proof}

\section{Numerical Series and the Cauchy Product}


A context where the previous construction appears naturally is the algebra of complex-valued sequences equipped with the Cauchy product defined~as
$$ \fall {n \in \mb N} (u \otimes v)_n = \sum_{k = 0}^n u_k v_{n - k} $$
In this context, the Mertens theorem states that given two convergent sequences $u$ and~$v$, if at least one of them is absolutely convergent, then their Cauchy product $u \otimes v$ is convergent and verifies
$$ \sum_{k = 0}^\infty (u \otimes v)_k = \Bigl(\sum_{i = 0}^\infty u_i \Bigr) \Bigl(\sum_{j = 0}^\infty v_j \Bigr) $$
Moreover, if both $u$ and $v$ are absolutely convergent, then so is $u \otimes v$.

Let now $\Sem$ (resp. $\Abs$) denote the set of convergent (resp. absolutely convergente) series and define:
$$ \fall{u \in \Sem} \quad \Sigma(u) = \sum_{k = 0}^\infty u_k $$
The Mertens theorem tells us that $\Abs$ is a unital subalgebra of $\mb C^{\mb N}$, that $\Sem$ is $\Abs$-stable, and that $\Sigma$ is an $\Abs$-form on $\Sem$.
It is then possible to define 
$$(\widetilde \Sem, \widetilde \Sigma) = \Ext_{\Abs}(\Sem, \Sigma)$$

\begin{proposition}
This extension is regular, linear and stable.
\end{proposition}
\begin{proof}
The regularity (which states that $\fall {\mb u \in \Sem} \widetilde \Sigma(\mb u) = \Sigma (\mb u)$) and linearity follow directly from propositions~\ref{prop:stability} and~\ref{prop:linearity}.
Stability, which states that
$$ (u_0, u_1, u_2, \ldots) \in \widetilde F \iff (0, u_0, u_1, u_2, \ldots) \in \widetilde F$$
and that they have the same sum, follows directly from the cancellation property: as $\varphi(e_1) \neq 0$, if
$$ u = (u_0, u_1, u_2, \ldots) \qquad \hbox{and} \qquad v = (0, u_0, u_1, u_2, \ldots) $$
then we have $e_1 \otimes u = v$, hence
$$ u \in \widetilde F \iff v = e_1 \otimes u \in \widetilde F $$
and, of course, $\tphi(v) = \tphi(e_1) \tphi(u) = \tphi(u)$.
\end{proof}

In the following, the extension of $\Sigma$ will also be denoted~$\Sigma$, dropping the tilde.

\subsection{Particular Sequences}

Let us now review some notable elements of~$\widetilde \Sem$ and, even, of~$\widetilde \Abs$ which, we recall, is a unital subalgebra of~$\mb C^{\mb N}$.

\begin{definition}[Geometric sequences]
For $\alpha \in \mb C$, let us define the geometric sequence
$$ G_\alpha = \bigl(\alpha^k\bigr)_{k \in \mb N} = (1, \alpha, \alpha^2, \alpha^3, \ldots). $$
\end{definition}

\begin{proposition}
For all $\alpha \neq 1$, $G_\alpha \in \widetilde \Abs$ with $\Sigma(G_\alpha) = \dfrac 1 {1 - \alpha}$.
\end{proposition}
\begin{proof}
This is a direct consequence of having $ G_\alpha \otimes (e_0 - \alpha e_1) = e_0 $.
\end{proof}

For $\alpha = -1$, we recognize Grandi's series, so that we have shown that
$$ G_{-1} \in \widetilde \Abs \qquad \hbox{with} \qquad \Sigma(G_{-1}) = \frac {\Sigma(e_0)}{\Sigma(e_0 + e_1)} = \frac 1 2 $$

\begin{proposition}
$$ G_1 \not \in \widetilde \Sem $$
\end{proposition}
\begin{proof} This follows from proposition~\ref{prop:not_in},~as
$$ G_1 \otimes (e_0 - e_1) = e_0 \in \Sem$$
with $\Sigma(e_0) = 1 \neq 0$ while $\Sigma(e_0 - e_1) = 0$.
\end{proof}

\begin{definition}
For $n \in \mb N$, let us define
\begin{gather*}
T_{ n} = \Biggl((-1)^k \binom {n + k} n \Biggr)_{k \in \mb N} = \Biggl((-1)^k \frac {(k+1)^{\overline n}}{n!} \Biggr)_{k \in \mb N} \\
{AP}_{\!n} = \Bigl((-1)^k (k+1)^n\Bigr)_{k \in \mb N} = \bigl(1, -2^n, 3^n, -4^n, 5^n, \ldots\bigr)
\end{gather*}
where $x^{\overline n}$ denotes the rising factorial of $x$ to the $n$~:
$$ x^{\overline n} = x \times (x + 1) \times \ \cdots \ \times (x + n - 1) $$
\end{definition}
It can be remarked that $T_0 = {AP}_{0} = \bigl((-1)^k)_{k \in \mb N} = G_{-1}$.



\begin{proposition}
For all $n \in \mb N$, we have $T_{n} \in \widetilde \Abs$ with $\Sigma(T_{ n}) = \frac 1 {2^{n + 1}}$.
\end{proposition}
\begin{proof}
This is a direct consequence of the fact that $\widetilde \Abs$ is stable by product and that
$$ T_{ n} = \bigotimes_{k = 0}^n G_{-1} $$
\end{proof}







Let $(B^+_n)$ denote the second Bernoulli numbers, and $\stirling n k$ the Stirling numbers of the second kind.

\begin{proposition}
For all $n \in \mb N$, $AP_{n} \in \widetilde \Abs$~with
$$ \Sigma(AP_{n}) = \frac {2^{n + 1} - 1}{n + 1} B^+_{n+1} $$
\end{proposition}
\begin{proof}
From the equality
$$ \fall {x \in \mb R} \fall {n \in \mb N} x^n = \sum_{k = 0}^n (-1)^{n - k} \stirling n k x^{\overline k}, $$
we directly deduce that
$$ \fall {n \in \mb N} AP_{n} = \sum_{k = 0}^n (-1)^{n - k} \, k! \, \stirling n k  \, T_k $$
so that $AP_{n} \in \widetilde \Abs$, and the value $\sum(AP_{n})$ follows from the representation of second Bernoulli numbers~$B^+_n$ using Worpitzky numbers~\cite{Worpitzky1883}:
$$ B^+_n = \frac n {2^{n + 1} - 2} \sum_{k = 0}^{n - 1} (-2)^{-k} k! \stirling n {k + 1} $$
\end{proof}

\begin{definition}[Powers]
For all $n \in \mb N$, we define
$$ P_n = \bigl((k + 1)^n\bigr)_{k \in \mb N} = (1, 2^n, 3^n, \ldots) $$

\end{definition}

\begin{proposition}
$$ P_1 \not \in \widetilde \Sem $$
\end{proposition}
\begin{proof}
We have $P_1 \otimes (e_0 - 2 e_1 + e_2) = e_0$, with $\varphi(e_0 - 2 e_1 + e_2) = 0$ and $\varphi(e_0) \neq 0$.
\end{proof}

The previous proposition shows that considering extension $\widetilde \Sem$ is not sufficient for affecting a sum to $P_1$. This is obviously not suprising as it is well know that a stable extension assigning a sum to $P_1$ would lead to inconsistencies such as~$1 = 0$. 
However, other extensions, based on other products, can be considered.

\section{A second product}


In this section, we will consider the following product:
$$ (u \circledast v)_n = \sum_{i j = n + 1} u_{i - 1} v_{j - 1} $$
In terms of $e_k$, this corresponds to having $e_i \circledast e_j = e_k$ with $k + 1 = (i + 1) (j + 1)$.
This product is associative and commutative, and has $e_0$ as neutral element. Moreover, the set $\Fin$ of finite sequences is a unital subalgebra of $(\mb C^{\mb N}, \circledast)$.










It is clear that if $x \in \Sem$ (resp. $\Abs$, $\Fin$) then so is $x \circledast e_k$ and hence, by linearity, that $\Sem$ (resp. $\Abs$, $\Fin$) is $\Fin$-stable w.r.t.\ $\circledast$ and we have $\Sigma(x \circledast e_k) = \Sigma(x) \Sigma(e_k)$.

\begin{proposition}
$\widetilde \Sem$ (resp. $\widetilde \Abs$) is $\Fin$-stable w.r.t.~$\circledast$ and $\tphi$ is a $\Fin$-form on $\widetilde \Sem$ with regard to $\circledast$.
\end{proposition}
\begin{proof}
Let is first remark that for all $i, j, k \in \mb N$,
$$(e_i \circledast e_k) \otimes (e_j \circledast e_k) = e_{(ik + i + k) + (jk + j + k)} = \bigl((e_i \otimes e_j) \circledast e_k\bigr) \otimes e_k$$
As a consequence, given $x \in \widetilde \Sem$ and $m \in \Abs$ such that $m \otimes x \in \Sem$, for all $k \in \mb N$, we have
$$ (m \circledast e_k) \otimes (x \circledast e_k) = \bigl((m \otimes x) \circledast e_k\bigr) \otimes e_k  \in \Sem $$
with $m \circledast e_k \in \Abs$, so that $x \circledast e_k \in \widetilde \Sem$. By linearity, for all $p \in \Fin$, one has $x \circledast p \in \widetilde \Sem$.
\end{proof}


This suggests to consider the extension of $\Sigma$ on $\widetilde \Sem$~to
$$ \Ext^{\circledast}_{\Fin}\bigl(\widetilde \Sem\bigr) $$
This extension is linear and preservative but it is not stable, as we will see after the next result.

\begin{proposition}
For all $n \in \mb N$, $P_n \in \Ext^{\circledast}_{\Fin}\bigl(\widetilde \Sem\bigr)$ with
$$ \Sigma(P_n) = - \frac {B^+_{n+1}}{n+1} = \zeta(-n) $$
\end{proposition}
\begin{proof}
We have
$$ P_n \circledast (e_0 - 2^{n+1} e_1) = (1, -2^n, 3^n, -4^n, 5^n, -6^n, \ldots) = AP_{n} $$
so that $P_n \in \Ext^{\circledast}_{\Fin}\bigl(\widetilde \Sem\bigr)$~and
$$ \Sigma(P_n) = \frac {\Sigma({AP}_{\!n})}{\Sigma(e_0 - 2^{n+1} e_1)} = - \frac {B^+_{n + 1}} {n + 1} = \zeta(-n)$$
\end{proof}








We thus have

\begin{align*}
1 + 1 + 1 + 1 + \cdots & = \Sigma(P_0) =  - \frac 1 2 &
1 + 2 + 3 + 4 + \cdots & = \Sigma(P_1) =  - \frac 1 {12} \\
1 + 4 + 9 + 16 + \cdots & = \Sigma(P_2) = \vphantom{- \frac 1 2} 0 &
1 + 8 + 27 + 64 + \cdots & = \Sigma(P_3) =  \frac 1 {120}
\end{align*}
and we can now rigorously express the chain of reasoning, presented in the introduction, that leads to the \emph{sum of all the integers}, i.e.\ to $S = \Sigma(P_1)$:
\begin{enumerate}
	\item $G_{-1} \otimes (e_0 + e_1) = e_0$ so that $G_{-1} \in \widetilde \Sem$ and
	$$A = \Sigma(G_{-1}) = \frac {\Sigma(e_0)}{\Sigma(e_0 + e_1)} = \frac 1 2~;$$
	\item $AP_1 \otimes (e_0 + e_1) = G_{-1}$ so that $AP_1 \in \widetilde \Sem$ and
	$$ B = \Sigma(AP_1) = \frac {\Sigma(G_{-1})}{\Sigma(e_0 + e_1)} = \frac 1 4~;$$
	\item $P_1 \circledast (e_0 - 4 e_1) = AP_1$ so that $P_1 \in \Ext_{\Fin}^\circledast(\widetilde \Sem)$~and
	$$ S = \Sigma(P_1) = \frac {\Sigma(AP_1)}{\Sigma(e_0 - 4 e_1)} = - \frac 1 {12} $$
\end{enumerate}


Since $\Ext_{\Fin}^\circledast(\widetilde \Sem)$ is based on the $\circledast$-product, its is irrelevant to consider the Cauchy product $u \otimes v$ of two sequences $u$ and $v$, unless they both belong to $\widetilde \Sem$ (and at least one belongs to $\widetilde \Abs$). Otherwise, even if $w = u \otimes v \in \Ext_{\Fin}^\circledast(\widetilde \Sem)$, it is irrelevant to see $w$ as $u \otimes v$ so that one need not have $\Sigma(w) = \Sigma(u) \Sigma(v)$.

For instance, we have $P_0, P_1 \in \Ext_{\Fin}^\circledast(\widetilde \Sem)$ and $P_1 = P_0 \otimes P_0$ but
$$\Sigma(P_1) = - \frac 1 {12} \neq \frac 1 4 = \Sigma(P_0)^2.$$
This also entails that stability -- which, in $\otimes$-extensions, was a direct consequence of the cancellation property -- is not a general property of the $\circledast$-extension. For instance, even though $\Ext_{\Fin}^\circledast(\widetilde \Sem)$ contains both $P_0 = (1, 1, 1, 1, \ldots)$ and  $P_0 \otimes e_1 = (0,1,1,1,1, \ldots)$, we have
$$ \Sigma(0,1,1,1,1, \ldots) \neq \Sigma(1, 1, 1, 1, \ldots) $$
since one has to write $(0,1,1,1,1, \ldots) = P_0 - e_0$ (rather than $P_0 \otimes e_1$) so that
$$ \Sigma(0,1,1,1,1, \ldots) = \Sigma(P_0) - \Sigma(e_0) = - \frac 3 2 $$
Similarly,
\begin{align*}
0 + 1 + 2 + 3 + 4 + 5 + \ \cdots \ & = \Sigma(P_1 - P_0) = \frac 5 {12} \\
0 + 0 + 1 + 2 + 3 + 4 + \ \cdots \ & = \Sigma(P_1 - 2 P_0 + e_0) = \frac {23} {12} \\
0 + 0 + 0 + 1 + 2 + 3 + \ \cdots \ & = \Sigma(P_1 - 3 P_0 + 2 e_0 + e_1) = \frac {53} {12} \\
\end{align*}

Let us show a few more examples of sum calculations:

\begin{align*}
1 + 2 + 3 + 4 + 5 + 6 + \ \cdots \ & = \Sigma(P_1) = - \frac 1 {12} \\
0 + 1 + 0 + 2 + 0 + 3 + \ \cdots \ & = \Sigma(P_1 \circledast e_1) = - \frac 1 {12} \\
1 + 0 + 2 + 0 + 3 + 0 + \ \cdots \ & = \frac 1 2 \Sigma\bigl((e_0 - 2 e_1) \circledast P_1 + (e_0 - e_1) \circledast P_0\bigr) \\ &  = + \frac 1 {24} \\[5pt]
1 + 3 + 5 + 7 + 9 + 11 + \ \cdots \ & = \Sigma(2 P_1 - P_0) = - \frac 2 3 \\
0 + 1 + 0 + 3 + 0 + 5 + \ \cdots \ & = \Sigma\bigl((2 P_1 - P_0) \circledast e_1\bigr) = - \frac 2 3 \\
1 + 0 + 3 + 0 + 5 + 0 + \ \cdots \ & = \Sigma\bigl(P_1 \circledast (e_0 - 2 e_1)\bigr) = + \frac 1 {12}
\end{align*}

Finally, we present a last result showing that it is not possible to assign a sum to the harmonic sequence in $\Ext_{\Fin}^{\circledast}(\widetilde \Sem)$.

\begin{proposition}
The harmonic sequence $H = \bigl(\frac 1 {n + 1}\bigr)_{n \in \mb N}$ is not in $\Ext_{\Fin}^\circledast(\widetilde \Sem)$.
\end{proposition}
\begin{proof}
We have
\begin{align*}
H \circledast (e_0 - e_1) & = \Bigl(1, \frac 1 2, \frac 1 3, \frac 1 4, \frac 1 5, \frac 1 6, \frac 1 7, \ldots\Bigr) \\
& - \Bigl(0, 1, 0, \frac 1 2, 0, \frac 1 3, 0, \ldots\Bigr) \\
& = \Bigl(1, - \frac 1 2, \frac 1 3, - \frac 1 4, \frac 1 5, - \frac 1 6, \frac 1 7, \ldots\Bigr) \in \widetilde \Sem
\end{align*}
with $\Sigma(e_0 - e_1) = 0$ while $\Sigma\bigl(H \circledast (e_0 - e_1)\bigr) = \ln 2 \neq 0$.
\end{proof}

\bibliographystyle{apalike}









\end{document}